\documentclass[11pt,reqno]{amsart}
\usepackage[T1]{fontenc}
\usepackage{lmodern}
\usepackage{mathtools,amssymb}
\usepackage[expansion=false]{microtype}
\usepackage{booktabs,array}
\usepackage{flafter}
\usepackage[margin=1.1in]{geometry}
\usepackage[hidelinks,unicode]{hyperref}
\hypersetup{pdftitle={Finite deletion-induced saturation for every non-complete graph},pdfauthor={Haochen Liu},pdfsubject={Graph theory; computer-assisted proof}}
\newtheorem{theorem}{Theorem}[section]
\newtheorem{lemma}[theorem]{Lemma}
\newtheorem{proposition}[theorem]{Proposition}
\newtheorem{corollary}[theorem]{Corollary}
\theoremstyle{definition}

\theoremstyle{remark}

\newcommand{\comp}[1]{\overline{#1}}
\newcommand{\Rclass}{\mathcal{R}}
\newcommand{\Dsat}{\mathsf{D}}
\newcommand{\Asat}{\mathsf{A}}

\newcommand{\Sym}{\mathbin{\triangle}}
\newcommand{\id}{\operatorname{id}}
\newcommand{\dom}{\operatorname{dom}}
\newcommand{\wt}{\operatorname{wt}}
\newcommand{\file}[1]{\texttt{\detokenize{#1}}}
\title[Finite deletion-induced saturation]{Finite deletion-induced saturation\break for every non-complete graph}
\author{Haochen Liu}
\address{School of Mathematics and Physics, Qingdao University of Science and Technology, Qingdao 266061, China}
\email{liuhc@mails.qust.edu.cn}
\email{1043564000@qq.com}
\date{September 18, 2026}

\begin{document}
\begin{abstract}
A graph $G$ is deletion-induced-saturated for $H$ if $G$ has an edge, contains no induced copy of $H$, and deleting any edge of $G$ creates an induced copy of $H$. We prove, with finite certificate verification, that a finite graph $H$ admits such a finite graph $G$ if and only if $H$ is not complete. This resolves the deletion conjecture of Fan, Hajebi, Hajebi and Spirkl. The main step transfers suitable free amalgamations to finite extensions using a local lifting theorem of Auinger, Bitterlich and Otto. A second criterion treats edge addition by protecting specified nonedges and then taking a maximal induced-$H$-free completion. Structural results of Bonamy, Groenland, Johnston, Morrison and Scott reduce the remaining targets to dense templates and a finite hereditary class. Two uniform constructions in halved cubes handle the dense templates. The finite part is supported by exhaustive coverage certificates, structural certificates and explicit hosts, including a circulant graph on $30$ vertices.
\end{abstract}
\maketitle

\section{Introduction}

Induced saturation asks for graphs that avoid a prescribed induced subgraph but acquire it after specified changes to their edge set. Martin and Smith~\cite{MS} introduced induced saturation in the language of trigraphs. For ordinary graphs, the symmetric problem requires both every edge deletion and every edge addition to create an induced copy of the target. Behrens, Erbes, Santana, Yager and Yeager~\cite{BE} constructed such hosts for several families, including odd cycles of length at least five. Axenovich and Csik\'os~\cite{AC} developed constructions using Cartesian products of cliques, while Dvo\v{r}\'ak~\cite{Dvorak} settled the path cases $P_n$ for all $n\ge6$.

In the asymmetric version, edge additions and edge deletions are considered separately. Fan, Hajebi, Hajebi and Spirkl~\cite{FHHScycles} proved the existence of finite deletion-saturated hosts for every even cycle. They subsequently formulated a general existence conjecture in~\cite[Conjecture~1.7]{FHHS}: every non-complete graph admits a finite graph which is induced-saturated under edge deletion alone. The exclusion of complete targets is necessary: deleting an edge cannot create a clique that was not already present.

For a finite graph $H$, write $\Dsat(H)$ if there exists a finite graph $G$ such that
\begin{equation}\label{eq:deletion-definition}
 E(G)\ne\varnothing,\qquad G\text{ is induced-}H\text{-free},\qquad
 G-e\text{ contains an induced }H\quad(e\in E(G)).
\end{equation}
Such a graph $G$ will be called a \emph{deletion host} for $H$. The qualification $E(G)\ne\varnothing$ rules out a vacuous witness. Our main result is the following.

\begin{theorem}\label{thm:main}
A finite simple graph $H$ has a finite deletion host if and only if $H$ is not complete.
\end{theorem}

Thus Theorem~\ref{thm:main} proves the finite asymmetric existence conjecture~\cite[Conjecture~1.7]{FHHS}. By complementation it also proves its equivalent edge-addition formulation. The two formulations allow different hosts; they impose no simultaneous addition-and-deletion requirement on a single host. The symmetric existence problem has different obstructions, as illustrated by the $P_4$ result of Martin and Smith~\cite{MS}.

The proof of sufficiency is computer-assisted. Its infinite-to-finite step is a mathematical existence argument, and its remaining finite cases are covered by explicit, exhaustively verified certificates. In particular, the general conclusion does not follow from checking all graphs up to a fixed order.

The finite extension theorem of Auinger, Bitterlich and Otto~\cite[Theorem~4.19]{ABO} supplies the principal tool. It belongs to the theory of extending partial automorphisms, developed for finite graphs by Hrushovski~\cite{Hrushovski} and for finite relational structures by Herwig~\cite{Herwig}, with further finite-extension results of Herwig and Lascar~\cite{HL}. We use the strengthened form with local right inverses to the quotient map from a free extension. These right inverses preserve induced substructures, a feature needed to retain an induced-subgraph exclusion after passage to a finite quotient. The extension of partial automorphisms alone does not provide this local induced-lifting conclusion. Two-point partial automorphisms simultaneously preserve a marked witness for every relevant pair. In the edge-addition direction, a second relation records protected nonedges; a maximal completion then treats all remaining nonedges.

Bonamy, Groenland, Johnston, Morrison and Scott~\cite{BGJMS} studied infinite induced-saturated graphs and established structural reductions based on $3^{*}$-cores. We use these structural results, together with the finite extension argument, to handle any graph for which either it or its complement contains a non-complete $3$-connected induced subgraph on at least five vertices. The remaining hereditary class has a particularly simple form from order $12$ onwards. Two constructions in halved cubes and several known deletion-host constructions settle this form. Finally, an exhaustive augmentation certificate handles orders $7$ through $11$; orders at most $6$ were settled in~\cite{FHHS}.

The paper is organized as follows. Section~\ref{sec:prelim} fixes notation and records the external graph-theoretic inputs. Sections~\ref{sec:finite}--\ref{sec:criteria} establish the finite protection criteria. Sections~\ref{sec:core} and~\ref{sec:dense} carry out the structural reduction and the dense-template analysis. Sections~\ref{sec:exceptions} and~\ref{sec:finitecases} describe the explicit exceptional hosts and the finite certificates. The proof of Theorem~\ref{thm:main} is assembled in Section~\ref{sec:conclusion}. Appendix~\ref{app:repro} specifies the reproducibility package.

\section{Notation and external inputs}\label{sec:prelim}

All graphs are simple and undirected. They are finite except for the free extensions used as intermediate objects. For $S\subseteq V(G)$, the notation $G[S]$ denotes the induced subgraph on $S$. Throughout, \emph{$H$-free} means \emph{induced-$H$-free}. We write $\comp G$ for the complement of $G$, $G\sqcup J$ for disjoint union, and $G\vee J$ for the join. A pair of vertices is treated as an unordered pair when it is an edge or a nonedge.

A \emph{vertex cut} of a connected graph is a set whose removal leaves at least two nonempty components. An \emph{adjacent two-vertex cut} is a cut consisting of the endpoints of an edge; a \emph{nonadjacent two-vertex cut} is defined analogously. A $3$-connected graph has at least four vertices and remains connected after deletion of any set of at most two vertices.

Write $\Asat(H)$ if a finite, non-complete, $H$-free graph $J$ exists such that $J+f$ contains an induced $H$ for every nonedge $f$ of $J$. We call $J$ an \emph{addition host}. Complementation gives
\begin{equation}\label{eq:duality}
 \Asat(H)\quad\Longleftrightarrow\quad\Dsat(\comp H).
\end{equation}
We use this equivalence in both directions, always keeping the exclusion of complete or edgeless hosts in the corresponding definition.

\begin{lemma}[Marked witnesses]\label{lem:marked}
If $G$ is $H$-free and $(G-xy)[S]\cong H$, then $x,y\in S$, and $G[S]$ is obtained from $H$ by adding precisely the edge corresponding to $xy$. The dual statement holds for edge addition.
\end{lemma}
\begin{proof}
If the two endpoints do not both lie in $S$, the modification does not change $G[S]$, contrary to $H$-freeness. Every other pair in $S$ is unchanged.
\end{proof}

We record the graph-theoretic results from~\cite{FHHS} that will be used below. Here $L(F)$ denotes the line graph of $F$.
\begin{enumerate}
\item Every non-complete graph on at most six vertices has a deletion host~\cite[Theorem~1.9]{FHHS}.
\item A graph with a governing block has a deletion host~\cite[Theorem~4.1]{FHHS}. In the application here, a block $B$ governs a target if it has a deletion host $W$ in which every other block of the target embeds inducedly with any prescribed vertex of that block mapped to any prescribed vertex of $W$.
\item Starting with a $2$-connected graph having a deletion host, successively adding vertices with at most one neighbor among the vertices already present preserves the existence of a deletion host~\cite[Corollary~4.2]{FHHS}.
\item A graph with two leaves at distance $2$ or $3$ has a deletion host~\cite[Theorem~5.1]{FHHS}.
\item Suppose $F$ has two nonadjacent leaves with distinct neighbors, neither lying in a component isomorphic to $K_{1,3}$. Then $\Dsat(L(F))$ holds~\cite[Theorem~5.2]{FHHS}. Such two leaves form a \emph{severed pair}.
\item If $H$ is not a line graph but $H+f$ is a line graph for some nonedge $f$, then $L(K_{2h})$ is a deletion host for $H$, where $h=|V(H)|$~\cite[Lemma~A.4(1)]{FHHS}.
\end{enumerate}
The finite extension theorem from~\cite{ABO} is stated in the form needed in Section~\ref{sec:finite}. The structural inputs from~\cite{BGJMS} are recorded in Section~\ref{sec:core}. For~\cite{ABO,BGJMS,FHHS}, all theorem locators refer to the specific arXiv versions listed in the bibliography, including where a journal version is also available.

\section{Finite extensions and induced lifting}\label{sec:finite}

We regard a graph as a relational structure with a symmetric binary edge relation $E$. At some stages there will also be a symmetric binary relation $N$ for protected nonedges. Initially these relations are irreflexive and disjoint. Membership in $N$ need not describe every nonedge in a later extension.

Let $A$ be a finite relational structure, and let $P$ be a finite list of partial automorphisms of $A$. Use an independent formal generator for each member of $P$, with inverses included as inverse symbols, and let $F$ be the resulting free group. The \emph{free extension} $T=A_PF$ has vertex set obtained from $V(A)\times F$ by the equivalence relation generated by
\begin{equation}\label{eq:free-identification}
 (x,pg)\sim(p(x),g)\qquad(x\in\dom(p),\ p\in P,\ g\in F).
\end{equation}
Each relation of $T$ consists of the images of the corresponding seed relation in all copies $A\times\{g\}$. Right multiplication on the group coordinate gives an action by automorphisms. The analogous construction for a suitable finite quotient group $Q$ is denoted by $A_PQ$.

For a finite relational structure $U$, its \emph{weight} is the number of elements occurring in no relation tuple, plus the total number of relation tuples. Tuples, including both orientations of a symmetric binary relation, are counted with respect to their relation symbols.

\begin{theorem}[Finite local lifting, {\cite[Theorem~4.19]{ABO}}]\label{thm:lifting}
For each positive integer $n$, there is a finite extension $B=A_PQ$ such that the seed $A$ embeds as an induced substructure, the specified partial automorphisms extend to automorphisms of $B$, and the natural map
\[
 \pi:T=A_PF\longrightarrow B=A_PQ
\]
is a relational homomorphism. Every weak substructure $U\subseteq B$ of weight at most $n$ admits a homomorphism $\phi:U\to T$ with
\begin{equation}\label{eq:right-inverse}
 \pi\phi=\id_U.
\end{equation}
Every relation tuple of $B$ lies in a $Q$-translate of a seed relation tuple.
\end{theorem}

Here a weak substructure may retain only some relation tuples supported on its vertex set. We apply the theorem to the full induced relational structure on that set.

\begin{lemma}[Induced lifting]\label{lem:induced-lifting}
If $U$ is the full induced substructure of $B$ on a vertex set of weight at most $n$, the map $\phi$ in Theorem~\ref{thm:lifting} is an induced embedding.
\end{lemma}
\begin{proof}
Equation~\eqref{eq:right-inverse} makes $\phi$ injective, and its homomorphism property preserves every relation tuple of $U$. If a tuple of elements of $\phi(U)$ belonged to a relation in $T$ while its preimage tuple did not belong to that relation in $U$, the homomorphism $\pi$ would send it to the missing tuple in $B$. This contradicts the assumption that $U$ contains all relations induced by $B$ on its vertices.
\end{proof}

For our applications, at most two binary relations are present and at most $h$ vertices must be lifted. If $r$ vertices occur in no relation tuple, then
\[
 \wt(U)\le r+2(h-r)^2\le 2h^2\qquad(h\ge1).
\]
This bound allows both loops and overlaps between the relations before they have been ruled out. We may therefore use
\begin{equation}\label{eq:weight-bound}
 n=2h^2+4.
\end{equation}
Once the free extension is known to be loopless and to have disjoint $E$ and $N$, lifting substructures on one or two vertices shows the same properties for $B$. More generally, an induced graph on at most $h$ vertices in the underlying $E$-graph of $B$ lifts inducedly to the underlying $E$-graph of $T$, by retaining the full induced structure in all relation symbols before applying the lemma.

\section{Two-point amalgamation and finite protection}\label{sec:criteria}

\subsection{The tree of seed copies}

Let $A$ be a finite graph and let $e_0$ be a specified adjacent pair or a specified nonadjacent pair. For every other pair $f$ of the same type, choose a bijection $p_f:e_0\to f$. These are partial automorphisms. When nonedges are protected, every seed pair has exactly one of the two relations $E,N$, and $p_f$ preserves both relations.

The left Cayley graph of the free group on the symbols $p_f$ is a tree, with edges joining $g$ to $p_fg$. In the free extension, the corresponding seed copies $A_g$ and $A_{p_fg}$ are identified along the two endpoints prescribed by $p_f$. The following properties justify the separator arguments used below.

\begin{lemma}\label{lem:tree-structure}
The seed copies in this two-point free extension satisfy the following properties.
\begin{enumerate}
\item Every $A_g$ embeds inducedly; distinct vertices within a seed copy are not identified.
\item The set of bags containing a fixed quotient vertex is a connected subtree of the index tree.
\item Cutting an edge of the index tree separates its bags into two sides whose vertex sets intersect only in the two-vertex seam. No graph edge joins vertices exclusive to opposite sides.
\item In the two-relation construction, $E$ and $N$ are irreflexive and disjoint throughout the free extension.
\end{enumerate}
\end{lemma}
\begin{proof}
An identification chain projects to a walk in the index tree. A chain starting and ending in the same bag projects to a closed walk. Successively cancelling backtracking pairs in this walk cancels mutually inverse partial bijections, so its initial and final coordinates agree. Thus no bag acquires identifications between distinct vertices.

If a quotient vertex lies in two bags, an identification chain joins those occurrences. Cancelling backtracking shows that it occurs in every bag along the unique tree path between them. This proves the connectedness assertion and also shows that a vertex appearing on both sides of a cut must belong to its seam.

Every graph edge comes from a single bag. Consequently an edge cannot join vertices exclusive to opposite sides. To see that a bag receives no additional internal relation, suppose a different bag supplies a relation between two of its vertices. Both vertices occur along the entire path joining the two bags. They therefore cross each seam together. The two-point identifications preserve the relation on that seam, so the relation was present in the original bag as well. The same argument for each of $E$ and $N$ proves inducedness and excludes conflicts; the absence of internal identifications excludes loops.
\end{proof}

\begin{lemma}[Bag exclusion]\label{lem:bag-exclusion}
Suppose $C$ is connected, has no cut vertex, and has no two-vertex cut of the seam type. If $A$ is $C$-free, then its two-point free extension is $C$-free.
\end{lemma}
\begin{proof}
Suppose an induced copy of $C$ exists in the extension. Choose a finite connected subtree with the fewest bags whose union covers its vertices. A single bag contradicts the $C$-freeness of $A$. Otherwise, cut the tree edge incident with a leaf bag. Minimality implies that the copy has a vertex exclusive to the leaf side and a vertex exclusive to the opposite side. Its intersection with the seam is therefore a separating set. By Lemma~\ref{lem:tree-structure}, this set has size zero, one, or two; in the last case its type agrees with the seam. These alternatives contradict, respectively, connectedness, the absence of a cut vertex, and the assumed exclusion of the relevant two-vertex cut. The separation is in the full index tree, so bags outside the chosen finite subtree cannot provide an edge between the two sides.
\end{proof}

\subsection{Completing protected nonedges}

\begin{lemma}[Protected completion]\label{lem:completion}
Let $E$ and $N$ be disjoint sets of unordered pairs on a finite set $V$, with $N\ne\varnothing$. Suppose $(V,E)$ is $H$-free. Assume that each $f\in N$ has a witness on a vertex set $S_f$ for which every pair is specified in $E$ or $N$, and adding $f$ to the graph on $S_f$ gives an induced copy of $H$. Then $\Asat(H)$ holds.
\end{lemma}
\begin{proof}
Choose a graph $J$ with the maximum number of edges among all $H$-free graphs satisfying
\[
 E\subseteq E(J)\subseteq\binom V2\setminus N.
\]
The set of choices is finite and nonempty. If a nonedge $f$ of $J$ lies in $N$, all pairs of the specified witness $S_f$ retain their original status, so $J+f$ contains an induced $H$. If $f\notin N$, then $J+f$ cannot be $H$-free, by maximality. Finally, $N\ne\varnothing$ ensures that $J$ is not complete.
\end{proof}

The argument does not require induced-$H$-freeness to be monotone under edge addition. An implementation which adds allowable edges must continue until none remain allowable; a single scan need not suffice.

\subsection{The two criteria}

\begin{theorem}[Deletion criterion]\label{thm:deletion-criterion}
Suppose $H$ contains a connected induced subgraph $C$ having neither a cut vertex nor an adjacent two-vertex cut. If $H$ has a nonedge $e_0$ such that $H+e_0$ is $C$-free, then $\Dsat(H)$ holds.
\end{theorem}
\begin{proof}
Put $A=H+e_0$ and map the marked edge $e_0$ to each other edge of $A$ by a two-point partial automorphism. The resulting free extension $T$ is $C$-free by Lemma~\ref{lem:bag-exclusion}. Apply Theorem~\ref{thm:lifting} with the bound~\eqref{eq:weight-bound}, where $h=|V(H)|$. An induced $C$ in the finite extension $B$ would lift inducedly to $T$, by Lemma~\ref{lem:induced-lifting}. Thus $B$ is $C$-free and hence $H$-free.

Every edge $e$ of $B$ is a translate of a seed edge $f$. The extended partial automorphism taking $e_0$ to $f$, followed by that translation, is an automorphism $\alpha$ of $B$ taking $e_0$ to $e$. For $f=e_0$, use the identity in the first step. Since $A$ embeds inducedly,
\[
 (B-e)[\alpha(V(A))]\cong A-e_0=H.
\]
This proves the required property for every edge of $B$. The seed edge $e_0$ ensures $E(B)\ne\varnothing$. If $A$ has only one edge, it already is the required host, so the construction with no generators causes no exception.
\end{proof}

\begin{theorem}[Addition criterion]\label{thm:addition-criterion}
Suppose $H$ contains a connected induced subgraph $C$ having neither a cut vertex nor a nonadjacent two-vertex cut. If $H$ has an edge $e_0$ such that $H-e_0$ is $C$-free, then $\Asat(H)$ holds.
\end{theorem}
\begin{proof}
Put $A=H-e_0$. Specify every pair of $A$ as an edge in $E$ or as a protected nonedge in $N$. Map the marked nonedge $e_0$ to every other member of $N$ by a two-point partial automorphism. The underlying $E$-graph of the free extension is $C$-free by Lemma~\ref{lem:bag-exclusion}, now with nonadjacent seams.

Theorem~\ref{thm:lifting} and Lemma~\ref{lem:induced-lifting} produce a finite extension $B$ with disjoint, irreflexive relations $E_B,N_B$, whose underlying $E_B$-graph is $C$-free and hence $H$-free. Each member $f$ of $N_B$ is a translate of a seed nonedge. As in the preceding proof, an automorphism maps $e_0$ to $f$ and transports the entire seed $A$ to a fully specified witness for adding $f$. Since $N_B$ contains the seed nonedge, Lemma~\ref{lem:completion} applies. If $A$ has only one nonedge, it is already an addition host.
\end{proof}

\begin{corollary}\label{cor:whole-graph}
Every connected non-complete graph with no cut vertex and no adjacent two-vertex cut has a finite deletion host. Every connected graph with at least one edge, no cut vertex and no nonadjacent two-vertex cut has a finite addition host.
\end{corollary}
\begin{proof}
Use $C=H$ in the relevant criterion. Changing one edge changes the edge count on the unchanged vertex set, so the modified seed cannot contain an induced $H$.
\end{proof}

In particular, complete graphs of order at least two are permitted in the addition criterion. The two criteria do not assert that a common host can satisfy both addition and deletion saturation.

\section{Reduction to a hereditary residual class}\label{sec:core}

\subsection{The $3^{*}$-core reduction}

A $3^{*}$-reduction repeatedly performs the following operations on the current induced graph, exhausting the first operation before applying the second:
\begin{enumerate}
\item delete a vertex of degree at most two;
\item delete adjacent vertices $a,b$ satisfying
\[
 N(a)=\{b,r,s\},\qquad N(b)=\{a,r,s\}
\]
for distinct $r,s$ outside $\{a,b\}$.
\end{enumerate}
The definition in~\cite{BGJMS} allows the deletion of a pair $a,b$ whenever $|N[a]\cup N[b]|\le4$. After the low-degree deletions have been exhausted, every vertex has degree at least three. An eligible pair then has $N[a]=N[b]$ of size four, so it is precisely an adjacent degree-three true-twin pair as in the second operation. Thus the terminal graph is the $3^{*}$-core in the terminology of~\cite{BGJMS}. For our direct argument, only the existence of a reduction sequence is needed.

\begin{lemma}\label{lem:core-destruction}
Suppose a $3^{*}$-reduction takes $H$ to a non-complete $3$-connected induced graph $C$. Adding any nonedge of $C$, or deleting any edge of $C$, makes the resulting graph on $V(H)$ $C$-free.
\end{lemma}
\begin{proof}
Let $H'$ be one of the two modified graphs. All changed pairs lie within the terminal $C$, so the degrees and prescribed neighborhoods of vertices at the moment they are removed in the original reduction sequence are unchanged in $H'$.

Any induced copy of $C$ has minimum degree at least three, and cannot contain a vertex removed by the first operation. At an operation of the second type, a copy containing just one of $a,b$ would give that vertex degree at most two. If it contained both, it would also have to contain $r,s$. Since a non-complete $3$-connected graph has at least five vertices, the copy would contain another vertex. Removing $r,s$ would separate $\{a,b\}$ from this other vertex, contrary to $3$-connectedness.

Thus any induced $C$ in $H'$ would survive the whole reduction and lie inside $C+e$ or $C-e$. These graphs have the same order as $C$ and a different number of edges, which is impossible.
\end{proof}

\begin{proposition}\label{prop:core-hosts}
If $H$ admits such a reduction to a non-complete $3$-connected graph, then both $\Dsat(H)$ and $\Asat(H)$ hold.
\end{proposition}
\begin{proof}
The graph $C$ has neither a cut vertex nor any two-vertex cut. Apply Lemma~\ref{lem:core-destruction} and the two finite protection criteria.
\end{proof}

We use the following two structural facts from~\cite[Lemmas~12--13]{BGJMS}:
\begin{enumerate}
\item every graph on at least $12$ vertices, or its complement, has a $3$-connected subgraph on at least five vertices;
\item if a graph or its complement contains a non-complete $3$-connected induced subgraph on at least five vertices, then the $3^{*}$-core of one of the two graphs is non-complete and $3$-connected.
\end{enumerate}
The second statement and complementation immediately yield the following.

\begin{corollary}\label{cor:outside-R}
If $H$ or $\comp H$ contains a non-complete $3$-connected induced subgraph on at least five vertices, then $\Dsat(H)$ holds.
\end{corollary}
\begin{proof}
In one orientation, Proposition~\ref{prop:core-hosts} gives both kinds of host. If this orientation is $\comp H$, use its addition host and~\eqref{eq:duality}.
\end{proof}

Let $\Rclass$ be the class of graphs $H$ such that neither $H$ nor $\comp H$ contains a non-complete $3$-connected induced subgraph on at least five vertices. This class is hereditary and closed under complementation. Only targets in $\Rclass$ remain.

\subsection{The structure from order twelve onwards}

We first state the finite fact used in the structural argument. Its complete certificate is specified in Appendix~\ref{app:repro}.

\begin{lemma}[Certified bipartite lemma]\label{lem:bipartite}
Let $(a,b)$ be one of $(5,7),(6,6),(7,5),(8,4),(9,3)$. Every bipartite graph with parts $A,B$, where $|A|=a$, $|B|=b$, and every vertex in $B$ has exactly $a-2$ neighbors in $A$, contains a $3$-connected subgraph on at least five vertices.
\end{lemma}
\begin{proof}[Certificate verification]
Keep $A$ labeled. Each vertex in $B$ is specified by its two missing neighbors, so, up to permutations of $B$, an input is a multiset of $b$ pairs from $\binom A2$. All such multisets are enumerated in lexicographic order. Their number is
\[
 \binom{\binom a2+b-1}{b}.
\]
For each input the certificate gives a vertex set of size at least five. The checker reconstructs the input, verifies the stated set, and tests connectedness of its induced graph after deleting every set of zero, one, or two vertices. The totals are given in Table~\ref{tab:bipartite}. All inputs have a verified witness.
\end{proof}

\begin{table}[htbp]
\centering
\caption{Complete enumeration for Lemma~\ref{lem:bipartite}.}\label{tab:bipartite}
\begin{tabular}{rrr}\toprule
$a$ & $b$ & Inputs with verified witnesses\\\midrule
5&7&11,440\\
6&6&38,760\\
7&5&53,130\\
8&4&31,465\\
9&3&8,436\\\midrule
\multicolumn{2}{l}{Total}&143,231\\\bottomrule
\end{tabular}
\end{table}

\begin{proposition}\label{prop:dense-reduction}
If $H\in\Rclass$ and $|V(H)|\ge12$, then one of $H,\comp H$ consists of a clique $C$ and at most two other vertices, each with at most two neighbors in $C$. Moreover, $|C|\ge10$.
\end{proposition}
\begin{proof}
By the first structural fact above, one orientation $D$ contains a $3$-connected subgraph on at least five vertices. Adding all edges induced on its vertex set preserves $3$-connectedness. Choose a $3$-connected induced subgraph of $D$ of maximum order. Since $D\in\Rclass$, it is a clique $C$, with $|C|\ge5$.

Every vertex outside $C$ has at most two neighbors in $C$. Otherwise, that vertex together with $C$ would induce a larger $3$-connected graph. Suppose there are at least three vertices outside $C$. If $|C|\ge9$, take $A\subseteq C$ of size $a=9$ and $B\subseteq V(D)\setminus C$ of size $b=3$. If $5\le|C|\le8$, take $A=C$, $a=|C|$, and choose $b=12-a$ vertices outside $C$, possible since $|V(D)|\ge12$.

In $\comp D$, every vertex in $B$ has at least $a-2$ neighbors in $A$. Retain exactly $a-2$ such edges at each vertex of $B$ and discard all within-part edges. Lemma~\ref{lem:bipartite} gives a $3$-connected subgraph on a vertex set $S$ of size at least five. The induced graph $\comp D[S]$ is still $3$-connected. A $3$-connected bipartite graph has at least three vertices in each part; in particular $S\cap A$ contains at least three vertices. These are independent in $\comp D$, so $\comp D[S]$ is not complete. This contradicts $D\in\Rclass$.

Consequently at most two vertices lie outside $C$, and $|C|\ge|V(H)|-2\ge10$.
\end{proof}

This is the structural argument underlying~\cite[proof of Theorem~10]{BGJMS}, with the finite bipartite step supplied here by a complete verifiable certificate.

\section{Halved cubes and dense templates}\label{sec:dense}

\subsection{Two uniform constructions}

The \emph{halved cube} $\Delta_n$ has as its vertices the even-cardinality subsets of $[n]$. Two vertices $X,Y$ are adjacent if and only if $|X\Sym Y|=2$. Thus $|V(\Delta_n)|=2^{n-1}$. Symmetric difference with a fixed even set, and permutation of the coordinates, are automorphisms. Together these act transitively on edges.

Let $T_M$ be the graph formed from a clique of order $M$ by adding one vertex with exactly two neighbors in the clique.

\begin{proposition}\label{prop:TM}
For every $M\ge4$, the graph $\Delta_M$ is a deletion host for $T_M$.
\end{proposition}
\begin{proof}
First let $M\ge5$. We classify the cliques of order at least five in $\Delta_M$. Translate a vertex of such a clique to $\varnothing$. Its other vertices are distinct pairwise-intersecting two-element sets. At least four such sets have a common element: if $\{p,q\}$ and $\{p,r\}$ occur, any member not containing $p$ must be $\{q,r\}$, and no fourth distinct pair can intersect all three. Translating back, every clique of order at least five lies in a star
\[
 \mathcal S_c=\{c\Sym\{i\}:i\in[M]\},
\]
where $c$ has odd cardinality. An $M$-clique is therefore a full star. For an even set $z$ outside this star, $|z\Sym c|$ is either $3$ or at least $5$. In the former case $z$ has exactly three neighbors in the star, and in the latter it has none. Hence $\Delta_M$ is $T_M$-free.

Delete the edge $e=\{\varnothing,\{1,2\}\}$. The star with odd center $\{3\}$ remains an $M$-clique. Before deletion, the vertex $\{1,2\}$ has precisely the three neighbors $\varnothing,\{1,3\},\{2,3\}$ in that star; afterwards it has exactly two. These vertices induce $T_M$ in $\Delta_M-e$. Edge transitivity treats every edge.

For $M=4$, we have $\Delta_4\cong K_{2,2,2,2}$. Every $4$-clique uses one vertex of each part, and any vertex outside it has exactly three neighbors in it. Thus $\Delta_4$ is $T_4$-free. The same marked-edge witness just given supplies a copy after deletion, again for every edge by transitivity. The explicit certificate also checks this boundary case directly.
\end{proof}

For $m\ge3$, set
\[
 J_m=K_2\vee(K_m\sqcup K_2).
\]

\begin{proposition}\label{prop:Jm}
For every $m\ge3$, the graph $\Delta_{m+3}$ is a deletion host for $J_m$.
\end{proposition}
\begin{proof}
The common neighborhood of any edge of $\Delta_n$ induces $K_2\mathbin{\square}K_{n-2}$. For the edge between $\varnothing$ and $\{1,2\}$, its common neighbors are
\[
 \{1,i\},\ \{2,i\}\qquad(3\le i\le n).
\]
These form two clique rows, with a matching between corresponding columns.

An induced $J_m$ would require an induced $K_m\sqcup K_2$ in the common neighborhood of the edge representing its two universal vertices. Since $m\ge3$, the $K_m$ must lie in one row. A $K_2$ anticomplete to it cannot be a matching edge; it must lie in the other row and use two columns avoided by the $K_m$. This requires $m+2$ columns, whereas $n=m+3$ gives only $m+1$ columns. Thus $\Delta_{m+3}$ is $J_m$-free.

Put $n=m+3$ and delete $e=\{\varnothing,\{1,2\}\}$. Consider
\begin{align*}
 S&=\{\{1,3\},\{2,3\}\},\\
 A&=\{\varnothing\}\cup\{\{3,k\}:5\le k\le n\},\\
 B&=\{\{1,2\},\{1,2,3,4\}\}.
\end{align*}
The set $S$ is a universal $K_2$ on these vertices, $A$ induces $K_m$, and $B$ induces $K_2$. Before deletion the only edge between $A$ and $B$ is $e$. This follows directly by taking symmetric differences. Hence these vertices induce $J_m$ after deletion, and edge transitivity completes the proof.
\end{proof}

\subsection{Exhausting the dense templates}

\begin{lemma}[Dense templates]\label{lem:dense}
Let $D$ consist of a clique $C$ of order $M\ge5$ and at most two other vertices, each with at most two neighbors in $C$. Then $\Dsat(\comp D)$ holds. If $D$ is non-complete, $\Dsat(D)$ also holds.
\end{lemma}
\begin{proof}
Every outside vertex has total degree at most three, so $C$ is the unique $M$-clique in $D$. Deleting an edge of $C$ destroys all copies of $K_M$. The addition criterion, with protector $K_M$, yields $\Asat(D)$, and therefore $\Dsat(\comp D)$.

It remains to consider a non-complete $D$. If there is one outside vertex, let $s$ be its number of neighbors in $C$. For $s=0$, the graph is the line graph of the disjoint union of an $M$-edge star and an independent edge. For $s=1$, it is the line graph of the tree obtained by extending one arm of an $M$-edge star by one edge. In each case the root has a severed pair, so~\cite[Theorem~5.2]{FHHS} applies. For $s=2$, use Proposition~\ref{prop:TM}.

Now suppose the outside vertices are $x,y$. If they are nonadjacent, every triangle of $\comp D$ contains the edge $xy$. Such a triangle exists, because $M\ge5$ and at most four clique vertices belong to $N_C(x)\cup N_C(y)$. Deleting $xy$ from $\comp D$ destroys every triangle. Apply the addition criterion to $\comp D$ with protector $K_3$, and then take complements.

Suppose $xy\in E(D)$, and put
\[
 s_x=|N_C(x)|\le s_y=|N_C(y)|,\qquad
 t=|N_C(x)\cap N_C(y)|.
\]
Table~\ref{tab:dense} lists all possible triples. We justify its entries below.

\begin{table}[htbp]
\centering
\caption{Adjacent outside vertices in Lemma~\ref{lem:dense}.}\label{tab:dense}
\begin{tabular}{@{}cl@{}}\toprule
$(s_x,s_y;t)$ & Reason for a deletion host\\\midrule
$(0,0;0)$ & Line graph with a severed pair\\
$(0,1;0)$ & Line graph with a severed pair\\
$(0,2;0)$ & $T_M$ with one leaf attached\\
$(1,1;1)$ & Line graph with a severed pair\\
$(1,1;0)$ & Deletion criterion with protector $C_4$\\
$(1,2;1)$ & Line graph with a severed pair\\
$(1,2;0)$ & Deletion criterion with protector $C_4$\\
$(2,2;0)$ or $(2,2;1)$ & Non-complete $3$-connected graph\\
$(2,2;2)$ & $J_{M-2}$, using Proposition~\ref{prop:Jm}\\\bottomrule
\end{tabular}
\end{table}

For the line-graph cases, represent the clique vertices by the edges $c\ell_i$ of an $M$-edge star. In case $(0,0;0)$, add a disjoint path with two edges. In case $(0,1;0)$, extend one star arm by two edges. In case $(1,1;1)$, attach two new leaves to a single star leaf. Finally, for $(1,2;1)$, write $N_C(x)=\{a\}$ and $N_C(y)=\{a,b\}$ and add root edges $\ell_a z$ and $\ell_a\ell_b$, corresponding to $x$ and $y$. These simple root graphs realize exactly the stated adjacencies. In the last case $z$ and any $\ell_j$ with $j\notin\{a,b\}$ form a severed pair. In the other cases choose a leaf on the new branch or component and a leaf on an unchanged star arm. The components involved are not claws. This verifies the hypotheses of~\cite[Theorem~5.2]{FHHS} in each case.

In case $(0,2;0)$, the graph is $T_M$ with a leaf attached at its degree-two vertex. The graph $T_M$ is $2$-connected. Proposition~\ref{prop:TM} and~\cite[Corollary~4.2]{FHHS} therefore apply.

In the two $C_4$ cases, write $N_C(x)=\{a\}$ with $a\notin N_C(y)$. For any $b\in N_C(y)$, the vertices $x,y,b,a$ induce a $4$-cycle in that order. Every induced $4$-cycle in $D$ must use both outside vertices and two clique vertices. Adding the nonedge $ay$ makes $N_C(x)\subseteq N_C(y)$ and destroys every such cycle. Since $C_4$ has neither a cut vertex nor an adjacent two-vertex cut, Theorem~\ref{thm:deletion-criterion} applies.

For $(2,2;0)$ and $(2,2;1)$, the graph $D$ is $3$-connected. After any two deletions the remaining clique is connected. If a surviving outside vertex loses both clique neighbors, both deletions were used on those neighbors; the other outside vertex survives and has a surviving clique neighbor because the two neighborhoods are different. The edge $xy$ connects the former vertex back to the clique. If one outside vertex is deleted, the remaining deletion cannot remove both clique neighbors of the other. Fewer deletions cause no difficulty. Since $D$ is non-complete, Corollary~\ref{cor:whole-graph} applies.

Finally, if $(s_x,s_y;t)=(2,2;2)$, the two common clique neighbors are universal and $D=J_{M-2}$. Here $M-2\ge3$, so Proposition~\ref{prop:Jm} applies. The table exhausts all allowed neighborhood sizes and intersections.
\end{proof}

\begin{corollary}\label{cor:large}
Every non-complete graph in $\Rclass$ on at least $12$ vertices has a finite deletion host.
\end{corollary}
\begin{proof}
Combine Proposition~\ref{prop:dense-reduction} with the two orientations of Lemma~\ref{lem:dense}.
\end{proof}

\section{Three exceptional targets}\label{sec:exceptions}

Let $P_4[a,b,c,d]$ denote the graph obtained from the consecutive vertices of a four-vertex path by replacing them with cliques of orders $a,b,c,d$, respectively, and joining consecutive cliques completely. The classification in the next section leaves three patterns requiring the constructions below.

\begin{proposition}\label{prop:exceptions}
Each of
\[
 P_4[2,2,1,2],\qquad P_4[2,2,2,1],\qquad P_4[2,2,2,2]
\]
has a finite deletion host.
\end{proposition}
\begin{proof}
The graph $P_4[2,2,1,2]$ is formed by identifying the degree-two vertex of $T_4$ with a vertex of a triangle. By Proposition~\ref{prop:TM}, $\Delta_4$ is a deletion host for $T_4$. Every vertex of $\Delta_4\cong K_{2,2,2,2}$ lies in a triangle, and any specified vertex of a triangle can be mapped to it. Thus $T_4$ is a governing block, and~\cite[Theorem~4.1]{FHHS} supplies the host.

For $P_4[2,2,2,1]$, the host is $\Delta_5$. Label its vertices $0,\ldots,15$ by increasing binary encodings of the even subsets of $[5]$. After deleting the edge $\{0,1\}$, the four consecutive cliques can be taken as
\begin{equation}\label{eq:halfcube5-witness}
 \{0,6\},\qquad\{2,3\},\qquad\{1,11\},\qquad\{13\}.
\end{equation}
The certificate enumerates all $\binom{16}{7}=11,440$ vertex subsets and finds no induced target before deletion. It also gives and checks a full witness for each of the $80$ edges.

For $P_4[2,2,2,2]$, use the circulant graph $G_{30}$ defined by
\begin{equation}\label{eq:circulant}
 V(G_{30})=\mathbb Z/30\mathbb Z,\qquad
 x\sim y\ \Longleftrightarrow\ x-y\in
 \{\pm1,\pm2,\pm5,\pm7,\pm8,\pm9,\pm10\}.
\end{equation}
This graph has $210$ edges. Exhaustive verification of its $\binom{30}{8}=5,852,925$ eight-vertex subsets finds no induced target. Table~\ref{tab:circulant} gives deletion witnesses for all seven distance classes; translation and reflection cover every edge. The certificate additionally expands and checks all $210$ edge witnesses, checking every vertex pair of each proposed induced copy.

For completeness, the exhaustive recognition in these two cases is described immediately after the proof, and the data and checker are specified in Appendix~\ref{app:repro}.
\end{proof}

\begin{table}[htbp]
\centering
\caption{Consecutive cliques of an induced $P_4[2,2,2,2]$ after deleting $\{0,s\}$ from $G_{30}$.}\label{tab:circulant}
\begin{tabular}{rcccc}\toprule
$s$ & First & Second & Third & Fourth\\\midrule
1 & $\{0,20\}$ & $\{22,29\}$ & $\{1,24\}$ & $\{3,26\}$\\
2 & $\{0,28\}$ & $\{20,21\}$ & $\{11,12\}$ & $\{2,4\}$\\
5 & $\{0,1\}$ & $\{21,22\}$ & $\{12,13\}$ & $\{5,4\}$\\
7 & $\{0,20\}$ & $\{28,29\}$ & $\{7,6\}$ & $\{14,16\}$\\
8 & $\{0,2\}$ & $\{22,23\}$ & $\{13,15\}$ & $\{8,6\}$\\
9 & $\{0,22\}$ & $\{1,2\}$ & $\{9,11\}$ & $\{16,18\}$\\
10 & $\{0,22\}$ & $\{1,2\}$ & $\{10,11\}$ & $\{18,19\}$\\\bottomrule
\end{tabular}
\end{table}

The explicit-host verifier is separate from the search code. For each vertex subset in the two path-blowup cases, it first checks the degree multiset. It then partitions the vertices by equality of their closed neighborhoods. These are the true-twin classes. An induced $P_4[2,2,2,2]$ has four such classes of size two, with quotient graph $P_4$. Conversely, those two properties characterize the target, since each class is a clique and adjacency between distinct classes is uniform. The same recognition applies to $P_4[2,2,2,1]$, with the singleton class at a path endpoint. The verifier checks $T_4$ in $\Delta_4$ by direct comparison over all vertex permutations of each five-vertex subset. Table~\ref{tab:hosts} gives the complete verification totals.

\begin{table}[htbp]
\centering
\caption{Exact verification of the three explicit base hosts. Every target count before deletion is zero.}\label{tab:hosts}
\begin{tabular}{@{}llrrr@{}}\toprule
Target & Host & Order & Subsets & Edge witnesses\\\midrule
$T_4$ & $\Delta_4$ & 8 & 56 & 24\\
$P_4[2,2,2,1]$ & $\Delta_5$ & 16 & 11,440 & 80\\
$P_4[2,2,2,2]$ & $G_{30}$ & 30 & 5,852,925 & 210\\\midrule
\multicolumn{3}{l}{Total} & 5,864,421 & 314\\\bottomrule
\end{tabular}
\end{table}

\section{Certified finite classification}\label{sec:finitecases}

By~\cite[Theorem~1.9]{FHHS}, targets on at most six vertices require no further analysis. Corollaries~\ref{cor:outside-R} and~\ref{cor:large} reduce the remaining problem to the non-complete members of $\Rclass$ on $7$ through $11$ vertices. We explain both why the enumeration covers every such graph and how each retained case is certified.

\subsection{Exhaustive hereditary augmentation}

Start with the one-vertex graph. At order $n$, for every retained representative at order $n-1$, add a new vertex with each of the $2^{n-1}$ possible neighborhoods. Reject a graph if it or its complement contains a non-complete $3$-connected induced subgraph on at least five vertices. Retain representatives of the remaining graphs up to isomorphism and complementation.

\begin{lemma}\label{lem:coverage}
This augmentation procedure represents every graph in $\Rclass$ at each order, up to isomorphism and complementation.
\end{lemma}
\begin{proof}
The assertion holds at order one. Let $H\in\Rclass$ have order $n$, and remove a vertex $v$. Heredity gives $H-v\in\Rclass$. By induction, either $H-v$ or its complement is isomorphic to a retained representative. In the former case one of the enumerated neighborhoods reconstructs $H$; in the latter one reconstructs $\comp H$. Since the class is complement-closed, this graph is not rejected. Its retained equivalence class therefore represents $H$.
\end{proof}

Canonical labeling is used to generate the data, but it is not trusted by the coverage verifier. For every parent graph and every neighborhood, the coverage certificate contains either a rejecting vertex set and orientation or the encoding of a retained representative. The verifier checks that every possible neighborhood occurs exactly once. It directly checks each rejecting witness and checks each accepted graph for isomorphism to the specified representative or its complement, using a separate isomorphism routine. It also checks membership of the retained representatives in $\Rclass$.

In total, $87,286$ augmentation records are verified: $83,518$ have explicit rejection witnesses and $3,768$ are accepted by isomorphism or complement-isomorphism. At orders two through seven, the retained lists are additionally compared with the complete small-graph atlas supplied by NetworkX~\cite{NetworkX}. This is a cross-check; the coverage argument itself is Lemma~\ref{lem:coverage} together with the per-augmentation certificates.

\subsection{Certificates for retained targets}

For each retained representative of order $7$ through $11$, both orientations are classified separately. Each non-complete orientation has one of the following certificates:
\begin{enumerate}
\item a dense-template decomposition as in Lemma~\ref{lem:dense}, in the target or its complement;
\item the hypotheses of one of the two protection criteria, either with the whole oriented graph as protector or with a specified induced protector and a specified modifying pair;
\item two leaves at distance $2$ or $3$;
\item a reduction by deleting vertices of degree at most one to a non-complete $2$-connected graph of order at most six;
\item a simple line-graph root, an explicit vertex-to-root-edge correspondence, and a severed pair;
\item a proof that the target is not a line graph and an explicit line-graph root after adding a specified nonedge;
\item an isomorphism to one of the three targets in Proposition~\ref{prop:exceptions}.
\end{enumerate}
An application of the addition criterion is made to the complement of the target and then converted by~\eqref{eq:duality}. The fourth rule uses the small-graph result and~\cite[Corollary~4.2]{FHHS}. The other external rules are precisely those recorded in Section~\ref{sec:prelim}.

The classification verifier directly tests all separator hypotheses. For a proper protector $C$, it enumerates every vertex subset of the modified seed of order $|V(C)|$ and tests induced isomorphism. This verifies that the edit destroys \emph{all} induced copies of $C$. A line-graph certificate is checked by comparing, for every pair of target vertices, adjacency in the target with intersection of the corresponding root edges.

The negative condition in the near-line-graph rule is checked by exhaustive clique-partition search. We include the justification to specify exactly what is verified.

\begin{lemma}[Clique-partition characterization]\label{lem:krausz}
A simple graph is a line graph of a simple graph if and only if its edges can be partitioned into cliques so that each vertex lies in at most two of the cliques.
\end{lemma}
\begin{proof}
In a line graph of a simple graph, the sets of edges incident with each root vertex give the required clique partition; cliques with fewer than two vertices may be omitted. Conversely, make a root vertex for each clique in such a partition. A vertex of the original graph belonging to two cliques becomes an edge joining their root vertices; one belonging to a single clique becomes an edge to a new leaf; one belonging to no clique becomes an isolated edge in the root. Two different original vertices cannot give parallel root edges, since their mutual edge would then belong to two partition cliques. The root is simple and has the original graph as its line graph.
\end{proof}

The verifier searches all eligible clique partitions. Failure to find one is therefore a complete check of the negative condition, rather than reliance on an unsuccessful heuristic recognizer.

\begin{proposition}[Certified residual classification]\label{prop:finite-classification}
Every non-complete member of $\Rclass$ on $7$ through $11$ vertices has a finite deletion host.
\end{proposition}
\begin{proof}[Certificate verification]
The coverage certificates establish Lemma~\ref{lem:coverage} through order $11$. The classification checker verifies a valid certificate of one of the preceding kinds for each non-complete orientation of each retained representative. The complete counts are in Table~\ref{tab:residual}. There are $638$ oriented rows, of which five are explicitly excluded complete graphs. Every one of the other $633$ rows has a verified certificate. Each rule implies the existence of a finite deletion host by the cited or proved results, including Proposition~\ref{prop:exceptions}. Hence every target covered by the augmentation has such a host.
\end{proof}

\begin{table}[htbp]
\centering
\caption{Complete classification of the residual class. Representatives are taken up to isomorphism and complementation; oriented rows are not asserted to be distinct isomorphism classes.}\label{tab:residual}
\begin{tabular}{rrrrr}\toprule
Order & Representatives & Oriented rows & Complete & Unresolved\\\midrule
7 & 142 & 284 & 1 & 0\\
8 & 101 & 202 & 1 & 0\\
9 & 31 & 62 & 1 & 0\\
10 & 23 & 46 & 1 & 0\\
11 & 22 & 44 & 1 & 0\\\midrule
Total & 319 & 638 & 5 & 0\\\bottomrule
\end{tabular}
\end{table}

\section{Proof of the main theorem and further questions}\label{sec:conclusion}

\begin{proof}[Proof of Theorem~\ref{thm:main}]
If $H$ is complete, an induced copy of $H$ appearing after an edge deletion would already have been a clique before deletion. Thus no $H$-free graph with an edge is a deletion host for $H$. This also excludes the complete graphs of orders zero and one under the usual induced-subgraph conventions.

Now suppose $H$ is non-complete, and put $h=|V(H)|$. If $h\le6$, use~\cite[Theorem~1.9]{FHHS}. Suppose $h\ge7$. If $H\notin\Rclass$, Corollary~\ref{cor:outside-R} applies. If $H\in\Rclass$ and $h\ge12$, use Corollary~\ref{cor:large}. The only remaining case is $H\in\Rclass$ with $7\le h\le11$, which is Proposition~\ref{prop:finite-classification}. These cases cover every finite non-complete graph, including disconnected graphs and graphs with isolated vertices.
\end{proof}

\begin{corollary}\label{cor:addition-main}
A finite simple graph has a finite addition host if and only if it has at least one edge.
\end{corollary}
\begin{proof}
Apply Theorem~\ref{thm:main} to its complement and use~\eqref{eq:duality}.
\end{proof}

The proof establishes finite existence without a practical upper bound on the order of a host. In particular, the finite groups furnished by Theorem~\ref{thm:lifting} are not constructed in the certificate computations. Explicit or quantitative bounds for these hosts would strengthen the result. It would also be useful to replace the residual classification, and especially the exceptional circulant construction, by a uniform graph-theoretic argument.

The role of the separator hypotheses is essential to the present method. For example, two $4$-cycles with edge sets
\[
 \{01,12,23,30\}\quad\text{and}\quad\{01,14,45,50\}
\]
are individually $P_4$-free, but their union along the common edge $01$ contains the induced path $2,1,4,5$. Thus arbitrary edge amalgamation does not preserve induced-subgraph exclusion. Similarly, a finite initial part of an infinite extension need not preserve a deletion witness at every edge. The finite lifting theorem and the marked automorphisms address these two separate requirements.

The computational proof uses exact discrete checks. The verifiers are separate from the generators and search routines, but this is algorithmic separation, not third-party peer review or formal verification in a proof assistant. The argument relies on the external theorems explicitly cited above and on the finite certificate checks described here.

\appendix
\section{Certificates and reproducibility}\label{app:repro}

The accompanying directory \file{certificates/} (distributed under \file{anc/} in the arXiv source package) contains the finite data, verification source code, execution results and a checksum manifest. Verification requires Python~3.10 or later, NetworkX, and a C++17 compiler available as \file{g++}. From that directory, run
\begin{quote}
\file{python3 run_all.py}
\end{quote}
with assertions enabled. The Python entry points reject optimized execution that disables assertions. The command verifies existing certificates; no search for a new host and no construction of a finite quotient group is needed.

\subsection{Graph encodings and coverage records}

A labeled graph on $\{0,\ldots,n-1\}$ is encoded by an integer whose successive bits correspond to
\[
 (0,1);\quad(0,2),(1,2);\quad(0,3),(1,3),(2,3);\quad\ldots.
\]
A bit equals one precisely when its pair is an edge. The files \file{residual_n.txt} store the retained encodings at each order.

Each line of \file{augmentation_coverage.jsonl} has the form
\[
 [n,\mathrm{parent},\mathrm{neighborhood},S,\mathrm{value}].
\]
The added vertex is $n-1$. If $S\ne0$, it encodes a rejecting induced vertex set and the value specifies the original or complement orientation. If $S=0$, the value is an accepted representative encoding. The verifier checks all parent-neighborhood pairs, all rejection hypotheses, and all accepted isomorphism claims. The resulting summary is \file{coverage_verification.json}.

\subsection{Structural and classification data}

Each line of \file{structural_bipartite_cert.txt} records $a$, the nondecreasing list of missing-pair indices for the $b$ vertices of the other part, and a bit mask for a $3$-connected induced vertex set. Pairs of the labeled $a$-set are indexed lexicographically. The checker independently regenerates the full multiset list and requires a valid certificate for every input.

The file \file{classification.json} records the order, representative, orientation, target encoding, proof rule and rule-specific data for all $638$ oriented targets. The checker does not import the classifier or the canonical-labeling routine. Its output is \file{classification_verification.json}.

\subsection{Explicit hosts and output validation}

The file \file{explicit_hosts.jsonl} contains the target edge set, the host edge set, and a full vertex map for every marked edge of each of the three hosts in Table~\ref{tab:hosts}. Each map is checked on all unordered vertex pairs, after deletion of its marked edge. The verifier requires all host edges to occur exactly once as marked edges.

The exhaustive C++ checker examines exactly the expected binomial number of vertex subsets for each of the two path-blowup hosts; the $T_4$ host is checked directly in Python. The hardened wrapper requires exactly one result for each C++ input host, the exact set of expected names, no duplicates, the correct subset count, and an integer zero for the number of induced target copies. It also requires a verified outcome for all three hosts. It rejects missing or empty subprocess output even when the subprocess exits successfully. This wrapper incorporates the output-validation repair identified in the supplied audit materials.

The four verification groups have the following totals:
\begin{center}
\begin{tabular}{@{}lr@{}}\toprule
Verification group & Checked objects\\\midrule
Bipartite structural inputs & 143,231\\
Augmentation records & 87,286\\
Oriented classification rows & 638\\
Explicit-host vertex subsets & 5,864,421\\
Explicit marked-edge witnesses & 314\\\bottomrule
\end{tabular}
\end{center}
The aggregate results, runtime versions and elapsed times are stored in \file{all_checks.json}, and a full run log is included. Checksum manifests describe the distributed files; rerunning the verification can update output logs and timing fields. Compiled executables are temporary products and are not mathematical inputs.


\begin{thebibliography}{99}
\raggedright

\bibitem{ABO}
K.~Auinger, J.~Bitterlich and M.~Otto,
\emph{Finite approximation of free groups II: The Theorems of Ash, Herwig--Lascar and Ribes--Zalesskii---revisited and strengthened},
preprint, arXiv:2507.11685v5 (2026).
\url{https://arxiv.org/abs/2507.11685v5}.

\bibitem{AC}
M.~Axenovich and M.~Csik\'os,
\emph{Induced saturation of graphs},
Discrete Math.\ \textbf{342} (2019), no.~4, 1195--1212.
\url{https://doi.org/10.1016/j.disc.2019.01.006}.

\bibitem{BE}
S.~Behrens, C.~Erbes, M.~Santana, D.~Yager and E.~Yeager,
\emph{Graphs with induced-saturation number zero},
Electron. J. Combin.\ \textbf{23} (2016), no.~1, Paper~P1.54, 23~pp.
\url{https://doi.org/10.37236/5095}.

\bibitem{BGJMS}
M.~Bonamy, C.~Groenland, T.~Johnston, N.~Morrison and A.~Scott,
\emph{Infinite induced-saturated graphs},
Canad. J. Math.\ (2026), First View, 1--31.
\url{https://doi.org/10.4153/S0008414X26102132}.
Result numbering in this paper follows the preprint,
\url{https://arxiv.org/abs/2506.08810v3}.

\bibitem{Dvorak}
V.~Dvo\v{r}\'ak,
\emph{$P_n$-induced-saturated graphs exist for all $n\ge6$},
Electron. J. Combin.\ \textbf{27} (2020), no.~4, Paper~P4.43, 6~pp.
\url{https://doi.org/10.37236/9579}.

\bibitem{FHHScycles}
X.~Fan, Sahab Hajebi, Sepehr Hajebi and S.~Spirkl,
\emph{Halfway to induced saturation for even cycles},
preprint, arXiv:2505.24100v2 (2025).
\url{https://arxiv.org/abs/2505.24100v2}.

\bibitem{FHHS}
X.~Fan, Sahab Hajebi, Sepehr Hajebi and S.~Spirkl,
\emph{Asymmetric induced saturation},
preprint, arXiv:2606.24763v1 (2026).
\url{https://arxiv.org/abs/2606.24763v1}.

\bibitem{NetworkX}
A.~A.~Hagberg, D.~A.~Schult and P.~J.~Swart,
\emph{Exploring network structure, dynamics, and function using NetworkX},
in Proceedings of the 7th Python in Science Conference,
G.~Varoquaux, T.~Vaught and J.~Millman (eds.),
Pasadena, CA, 2008, pp.~11--15.
\url{https://doi.org/10.25080/TCWV9851}.

\bibitem{Herwig}
B.~Herwig,
\emph{Extending partial isomorphisms on finite structures},
Combinatorica \textbf{15} (1995), no.~3, 365--371.
\url{https://doi.org/10.1007/BF01299742}.

\bibitem{HL}
B.~Herwig and D.~Lascar,
\emph{Extending partial automorphisms and the profinite topology on free groups},
Trans. Amer. Math. Soc.\ \textbf{352} (2000), no.~5, 1985--2021.
\url{https://doi.org/10.1090/S0002-9947-99-02374-0}.

\bibitem{Hrushovski}
E.~Hrushovski,
\emph{Extending partial isomorphisms of graphs},
Combinatorica \textbf{12} (1992), no.~4, 411--416.
\url{https://doi.org/10.1007/BF01305233}.

\bibitem{MS}
R.~R.~Martin and J.~J.~Smith,
\emph{Induced saturation number},
Discrete Math.\ \textbf{312} (2012), no.~21, 3096--3106.
\url{https://doi.org/10.1016/j.disc.2012.06.015}.

\end{thebibliography}
\end{document}